\documentclass[10pt]{amsart}

\usepackage{amsmath, amsthm, amscd, amsfonts, amssymb, graphicx, color}
\usepackage{latexsym}
\usepackage{amscd}
\usepackage{amsthm}
\usepackage{amssymb}
\usepackage{enumerate}
\usepackage{mathabx}

\theoremstyle{plain}
\newtheorem{theorem}{Theorem}

\newtheorem{lemma}[theorem]{Lemma}
\newtheorem{prop}[theorem]{Proposition}
\theoremstyle{remark}
\theoremstyle{definition}

\newtheorem{definition}[theorem]{Definition}
\newtheorem{example}[theorem]{Example}

\newcommand{\R}{\mathbb{R}}

\begin{document}

\title{Characterization of distributions whose forward differences are exponential polynomials}
\author{J.~M.~Almira}



\address{Departamento de Matem\'{a}ticas, Universidad de Ja\'{e}n, E.P.S. Linares,  Campus Cient\'{\i}fico Tecnol\'{o}gico de Linares, 23700 Linares, Spain}
\email{jmalmira@ujaen.es}

\begin{abstract}
Given $\{h_1,\cdots,h_{t}\} $ a finite subset  of $\mathbb{R}^d$, we study the continuous complex valued functions and the Schwartz complex valued distributions $f$ defined on $\mathbb{R}^d$ with the property that the forward differences  $\Delta_{h_k}^{m_k}f$ are (in distributional sense) continuous exponential polynomials for some natural numbers $m_1,\cdots,m_t$. 
\end{abstract}

\maketitle

\markboth{J.~M.~Almira}{Distributions whose forward differences are exponential polynomials}

\section{Introduction}

Let  $X_d$ indistinctly denote  either the set of continuous complex valued functions $C(\mathbb{R}^d)$ or the set of Schwartz complex valued distributions 
$\mathcal{D}(\mathbb{R}^d)'$. Let  $f\in X_d$ and let us denote by $\tau_y$ and $\Delta_h^m$  the translation  operator and the forward differences operator defined on $X_d$, respectively. In formulas,  
$(\tau_yf)(x)=f(x+y)$ and $(\Delta_h^mf)(x)=\sum_{k=0}^{m}\binom{m}{k}(-1)^{m-k}\tau_{kh}(f)$ if $f$ is an ordinary function, and $\tau_yf\{\varphi\}=f\{\tau_{-y}(\varphi)\}$, $(\Delta_h^mf)\{\varphi\}=f\{\Delta_{-h}^m(\varphi)\}$ if $f\in\mathcal{D}(\mathbb{R}^d)'$, 
$\varphi\in \mathcal{D}(\mathbb{R}^d)$.

We prove that, if $\{h_1,\cdots,h_{t}\} $ spans a dense subgroup of $\mathbb{R}^d$, $f\in X_d$ and there exist natural numbers 
$\{m_k\}_{k=1}^t$ such that, for every $k\in \{1,...,t\}$,  $\Delta_{h_k}^{m_k}f$ is (in distributional sense) a continuous exponential polynomial,  then $f$ is (in distributional sense) a continuous exponential polynomial. Moreover, for the case of continuous functions, we characterize the functions $f$ satisfying that $\Delta_{h_k}^{m_k}f$ is an exponential polynomial for arbitrary sets  $\{h_1,\cdots,h_{t}\} $.

\section{The case of finitely generated dense subgroups of $\mathbb{R}^d$}
Let us state two technical results, which are  important for our arguments in this section. These results were, indeed, recently introduced by the author, and have  proved  their usefulness for the study of several Montel-type theorems for polynomial and exponential polynomial functions (see, e.g., \cite{AA_AM}-\cite{AS_DM}). We include the proofs for the sake of completeness.

\begin{definition} Let $t$ be a positive integer, $E$ a vector space, $L_1,L_2,\cdots,L_t:E\to E$ pairwise commuting linear operators. Given a subspace $V\subseteq E$, we denote by   $\diamond_{L_1,L_2,\cdots,L_t}(V)$ the smallest subspace of $E$ containing $V$ which is $L_i$-invariant for $i=1,2,\dots,t$. \end{definition}

\begin{lemma} \label{uno} With the notation we have just introduced,  if $V$ is an $L^n$-invariant subspace of $E$, then the linear space
\begin{equation*}
V_L^{[n]}=V+L(V)+\dots +L^n(V)\,.
\end{equation*}
is $L$-invariant. Furthermore $V_L^{[n]} =  \diamond_{L}(V)$. In other words, $V_L^{[n]}$ is the smallest $L$-invariant subspace of $E$ containing $V$.
\end{lemma}

\begin{proof}
Let $v$ be in $V_L^{[n]}$, then
\begin{equation}\label{Box}
v=v_0+Lv_1+\dots+L^{n-1}v_{n-1}+L^nv_n
\end{equation}
with some elements $v_0,v_1,\dots,v_n$ in $V$. By the $L^n$-invariance of $V$, we have that $L^nv_n=u$ is in $V$, hence it follows
\begin{equation*}
Lv=L(v_0+u)+L^2v_1+\dots+L^{n}v_{n-1}\,,
\end{equation*}
and the right hand side is clearly in $V_L^{[n]}$. This proves that $V_L^{[n]}$ is $L$-invariant. On the other hand, if $W$ is an $L$-invariant subspace of $E$, which contains $V$, then $L^k(V)\subseteq W$ for $k=1,2,\dots,n$, hence the right hand side of \eqref{Box} is in $W$.
\end{proof}

\begin{lemma}\label{dos}
Let $t$ be a positive integer, $E$ a vector space, $L_1,L_2,\cdots,L_t:E\to E$ pairwise commuting linear operators, and let $s_1,\cdots, s_t$ be natural numbers. Given a subspace $V\subseteq E$ we form the sequence of subspaces
\begin{equation}\label{rec}
V_0=V,\enskip V_{i}=(V_{i-1})_{L_{i}}^{[s_{i}]},\enskip i=1,2,\dots,t\,.
\end{equation}
If  for $i=1,2,\dots,t$ the subspace $V$ is $L_i^{s_i}$-invariant, then $V_t$ is $L_i$-invariant, and it contains $V$. Furthermore $V_t= \diamond_{L_1,L_2,\cdots,L_t}(V)$ and $\dim(V)<\infty$  if and only if $\dim(\diamond_{L_1,L_2,\cdots,L_t}(V))<\infty$.
\end{lemma}

\begin{proof}
First we prove by induction on $i$ that $V_{i}$ is $L_{j}^{s_{j}}$-invariant and it contains $V$ for each $i=0,1,\dots,t$ and $j=1,2,\dots,t$. For $i=0$ we have $V_0=V$, which is  $L_j^{s_j}$-invariant for $j=1,2,\dots,t$, by assumption.
\vskip.3cm

Suppose that $i\geq 1$, and we have proved the statement for $V_{i-1}$. Now we prove it for $V_{i}$. If $v$ is in $V_{i}$, then we have
\begin{equation*}
v=u_0+L_{i}u_1+\dots+L_{i}^{s_{i}}u_{s_i}\,,
\end{equation*}
where $u_j$ is in $V_{i-1}$ for $j=0,1,\dots,s_{i}$. It follows for $j=1,2,\dots,t$
\begin{equation*}
L_{j}^{s_{j}}v=(L_{j}^{s_{j}}u_0)+L_{i}(L_{j}^{s_{j}}u_1)+\dots+L_{i}^{s_{i}}(L_{j}^{s_{j}}u_{s_i})\,.
\end{equation*}
Here we used the commuting property of the given operators, which obviously holds for their powers, too. By the induction hypothesis, the elements in the brackets on the right hand side belong to $V_{i-1}$, hence $L_{j}^{s_{j}}v$ is in $V_i$, that is, $V_i$ is $L_{j}^{s_{j}}$-invariant. As $V_i$ includes $V_{i-1}$, we also conclude that $V$ is in $V_i$, and our statement is proved.
\vskip.3cm

Now we have
\begin{equation*}
V_t=V_{t-1}+L_t(V_{t-1})+\dots +L_t^{s_t}(V_{t-1})\,,
\end{equation*}
and we apply the previous lemma: as $V_{t-1}$ is $L_t^{s_t}$-invariant, we have that $V_t$ is $L_t$-invariant.
\vskip.3cm

Let us now prove the invariance of $V_t$
under the operators $L_j$ ($j<t$). Since $V_1$
is clearly $L_1$-invariant, by Lemma 12, an induction process
gives that $V_{t-1}$ is $L_i$-invariant for $1\leq i\leq t-1$. Thus, if we
take $1\leq i\leq t-1$, then we can use that $L_iL_t=L_tL_i$  and $L_i(V_{t-1})$ is a subset
of $V_{t-1}$ to conclude that
\begin{eqnarray*}
L_i(V_t)
&=& L_i(V_{t-1})+L_t(L_i(V_{t-1}))+\cdots +L_t^{s_t}(L_i(V_{t-1})) \\
&\subseteq & V_{t-1}+L_t(V_{t-1})+\cdots+L_t^{s_t}(V_{t-1}) =V_t\,,
\end{eqnarray*}
which completes this part of the proof.

Suppose that $W$ is a subspace in $E$ such that $V\subseteq W$, and $W$ is $L_j$-invariant for $j=1,2,\dots,t$. Then, obviously, all the subspaces $V_i$ for $i=1,2,\dots,t$ are included in $W$. In particular, $V_t$ is included in $W$. This proves that $V_t$ is the smallest subspace in $E$, which includes $V$, and which is invariant with respect to the family of operators $L_i$. In particular,  $V_t=\diamond_{L_1,L_2,\cdots,L_t}(V)$ is uniquely determined by $V$, and by the family of the operators $L_i$, no matter how we label these operators.
\end{proof}

The following result generalizes Anselone-Koreevar's Theorem \cite{anselone}.

\begin{lemma}\label{VdentroW} Assume that $W$ is a finite dimensional subspace of $X_d$, $\{h_1,\cdots,h_{t}\}\subset \mathbb{R}^d$  and $h_1\mathbb{Z}+h_2\mathbb{Z}+\cdots+h_t\mathbb{Z}$ is dense in $\mathbb{R}^d$. Assume, furthermore, that  $\Delta_{h_k}^{m_k}(W)\subseteq H$, $k=1,\cdots,t$,  for certain positive integral numbers $m_k$  and certain finite dimensional subspace $H$ of $X_d$ satisfying $$\bigcup_{k=1}^t \Delta_{h_k}^{m_k}(H) \subseteq H.$$  Then there exists a
finite dimensional subspace $Z$ of  $X_d$ which is invariant by translations and contains $W\cup H$. Consequently, all elements of $W\cup H$ are continuous exponential polynomials.
\end{lemma}

\begin{proof} We apply Lemma \ref{dos} with $E=X_d$, $L_i=\Delta_{h_i}$, $s_i=m_i$,  $i=1,\cdots,t$, and $V=W+H$, since $$\Delta_{h_i}^{m_i}(W+H)=\Delta_{h_i}^{m_i}(W)+ \Delta_{h_i}^{m_i}(H)\subseteq H+\Delta_{h_i}^{m_i}(H)\subseteq H\subseteq W+H,$$ so that
$V\subseteq Z=\diamond_{\Delta_{h_1},\Delta_{h_2},\cdots,\Delta{h_t}}(V)$ and $Z$ is a finite dimensional subspace of $X_d$ satisfying $\Delta_{h_i}(Z)\subseteq Z$ , $i=1,2,\cdots,t$. Hence $Z$ is invariant by translations, since $h_1\mathbb{Z}+h_2\mathbb{Z}+\cdots+h_t\mathbb{Z}$ is dense in $\mathbb{R}^d$. Applying  Anselone-Koreevar's theorem, we conclude that  all elements of $Z$ (hence, also all elements of $W\cup H$) are continuous exponential polynomials.
\end{proof}

Now we can demonstrate the main result of this section:

\begin{theorem} \label{EP} Assume that $\{h_1,\cdots,h_{t}\} $ spans a dense subgroup of $\mathbb{R}^d$, $f\in X_d$ and there exist natural numbers 
$\{m_k\}_{k=1}^t$ such that, for every $k\in \{1,...,t\}$,  $\Delta_{h_k}^{m_k}f$ is a continuous exponential polynomial. Then $f$ is a continuous exponential polynomial.
\end{theorem}

\begin{proof}
Let $g_k=\Delta_{h_k}^{m_k}f$ be an exponential polynomial for $k=1,\cdots, t$. Then we can take, in  Lemma \ref{VdentroW},  $W=\mathbf{span}\{f\}$ and $H=\tau(\mathbf{span}\{g_k\}_{k=1}^t)$, the smallest subspace of $X_d$ which is translation invariant and contains $\{g_k\}_{k=1}^t$, since $H$ is finite dimensional.
\end{proof}

\section{Finitely generated nondense subgroups of $\mathbb{R}^d$}

In this section we demonstrate that density of $G=h_1\mathbb{Z}+h_2\mathbb{Z}+\cdots+h_t\mathbb{Z}$  in $\mathbb{R}^d$ is a necessary  hypothesis in Theorem \ref{EP}. Moreover, under the hypothesis that this group $G$ is not  a dense subgroup of $\R^d$, we characterize the continuous functions $f$ satisfying that, for a certain finite dimensional space $H\subseteq C(\mathbb{R}^d)$ and certain natural numbers  $n_k, m_k$, $k=1,\cdots, t$, the relations $\bigcup_{k=1}^t \Delta_{h_k}^{m_k}(H) \subseteq H$ and $\Delta_{h_k}^{n_k}f\in H$,  $k=1,\cdots,t$, hold.

For $d=1$, the condition that $\{h_1,....,h_t\}$ spans a dense subgroup of $\mathbb{R}^d$ can be resumed to: $t\geq 2$ and $h_i/h_j\not\in\mathbb{Q}$ for some $i\neq j$.  We prove that, if   $h_1/h_2 \in\mathbb{Q}$ and $m\geq 1$ then there are finite dimensional subspaces $V,H$ of $C(\mathbb{R})$ such that $\Delta_{h_1}^m(V)\cup \Delta_{h_2}^m(V) \subseteq H$, $H$ is $\Delta_{h_k}^{m}$-invariant (indeed, it can be chosen translation invariant), and no finite dimensional translation invariant subspace $W$ of  $C(\mathbb{R})$  satisfies $V\subseteq W$. To construct these spaces we need to use the following technical results:
\begin{lemma}\label{propiedaddiferencias} Assume that $f:\mathbb{R}\to \mathbb{C}$ satisfies $\Delta_h^mf=0$, and let $p\in \mathbb{Z}$. Then $\Delta_{ph}^mf=0$.
\end{lemma}
\begin{proof} For $m=1$ the result is trivial, since the periods of the function $f$ form an additive subgroup of $\mathbb{R}$. For $m\geq 2$ the result follows from commutativity of the composition of the operators $\Delta_h$ (i.e., we use that $\Delta_h\Delta_kf=\Delta_k\Delta_hf$). Concretely, $\Delta_h^mf=\Delta_h(\Delta_h^{m-1}f)=0$ implies that $h$ is a period of $\Delta_h^{m-1}f$. Hence $ph$ is also a period of this function and $\Delta_{ph}(\Delta_h^{m-1}f)=0$. Now we use that $\Delta_{ph}(\Delta_h^{m-1}f)=\Delta_{h}(\Delta_{ph}\Delta_h^{m-2}f)$ and iterate the argument several times to conclude that $\Delta_{ph}^mf=0$.
\end{proof}

\begin{lemma} \label{ecuaciondiferencias} Let $h>0$ and assume that $g\in C(\mathbb{R})$ satisfies $g(h\mathbb{Z})=\{0\}$. Then there exists $f\in C(\mathbb{R})$ such that $f(h\mathbb{Z})=\{0\}$ and $\Delta_hf=g$. Consequently, for each $m\geq 1$ there exists $F_m\in C(\mathbb{R})$ such that $F_m(h\mathbb{Z})=\{0\}$ and $\Delta_h^mF_m=g$.
\end{lemma}

\begin{proof} Given $g\in C(\mathbb{R})$ satisfying $g(h\mathbb{Z})=\{0\}$, it is easy to check that the function
\[
f(z)=\left \{
\begin{array}{cccccc}
\sum_{j=0}^{k-1}g(x+jh) & \text{if} & z=x+kh, \ k\in\mathbb{N}\setminus\{0\}, \text{ and } x\in [0,h) \\
-\sum_{j=1}^{k}g(x-jh) & \text{if} & z=x-kh, \ k\in\mathbb{N}\setminus\{0\}, \text{ and } x\in [0,h) \\
0 & \text{if} & z\in [0,h) \\
\end{array} \right.
\]
is continuous and satisfies $f(h\mathbb{Z})=\{0\}$ and $\Delta_hf=g$. The second claim of the lemma follows by iteration of this argument.
\end{proof}
\begin{example} \label{ejemp}
Let us now assume that $h_1,h_2>0$, $h_1/h_2\in\mathbb{Q}$ , and $m\in\mathbb{N}$ ($m\geq 1$). Obviously, there exists $h>0$ and $p,q\in\mathbb{N}$ such that $h_1=ph$ and $h_2=qh$. Consider the function $\phi\in C(\mathbb{R})$ defined by $\phi(x)=|x|$ for $|x|\leq h/2$ and $\phi(x+h)=\phi(x)$ for all $x\in\mathbb{R}$ and use Lemma \ref{ecuaciondiferencias} with $g=\phi$ to construct, for each $m\geq 2$, a function $f_m\in C(\mathbb{R})$ such that  $\Delta_h^{m-1}f_m=\phi$. Take $f_1=\phi$. Then $\Delta_h^{m}f_m=\Delta_h\phi=0$ for all $m$. This, in conjunction with Lemma \ref{propiedaddiferencias}, implies that the space $V_m=\mathbf{span}\{f_{m}\}$ satisfies   $\Delta_{h_1}^{m}(V_m)\cup \Delta_{h_2}^{m}(V_m)=\{0\} $. Obviously, $H=\{0\}$ is translation invariant. On the other hand, $V_m$ is not a subset of any finite dimensional translation invariant subspace of $C(\mathbb{R})$, since $f_m\in V_m$ is not an exponential polynomial (indeed, it is not a differentiable function).
\end{example}

Let $d$ be a positive integer. If $G$ denotes the additive subgroup of $\mathbb{R}^d$ generated by the elements
$\{h_1,\cdots,h_t\}$, then it is well-known \cite[Theorem 3.1]{W} that $\overline{G}$, the topological closure of $G$ with the Euclidean topology, satisfies  $\overline{G}=V\oplus \Lambda$, where
$V$ is a vector subspace of $\mathbb{R}^d$ and $\Lambda$ is a discrete additive subgroup of $\mathbb{R}^d$. Furthermore, the case when $G$ is dense in $\mathbb{R}^d$, or, what is the same, the case whenever $V=\mathbb{R}^d$, has been characterized in several different ways (see e.g., \cite[Theorem 442, page 382]{HW},  \cite[Proposition 4.3]{W}).  In the next  proposition  we prove that, under the notation we have just introduced, if $V$ is a proper vector subspace of $\mathbb{R}^d$,  there are continuous functions $f$  satisfying that $\Delta_{h_k}^{m_k}f$ is a continuous exponential polynomial for $k=1,2,\cdots,t$ but $f$ is not an exponential polynomial.

\begin{prop} \label{otro}
Let $t$ be a positive integer, let $n_1,n_2,\dots,n_t$ be natural numbers, and assume that the subgroup $G$ of $\R^d$ generated by $\{h_1,h_2,\dots,h_t\}$ satisfies $\overline{G}=V\oplus \Lambda$, where $V$ is a proper vector subspace of $\mathbb{R}^d$, and $\Lambda$ is a discrete additive subgroup of $\mathbb{R}^d$. Then there exist $H$ linear finite dimensional subspace of $C(\mathbb{R}^d)$ satisfying
$$\bigcup_{k=1}^t \Delta_{h_k}(H) \subseteq H,$$
and  $\varphi:\mathbb{R}^d\to\mathbb{R}$, continuous function satisfying
\[
\Delta_{h_k}^{n_k}\varphi \in H
\]
for $k=1,\cdots,t$, such that  $\varphi$ is not an exponential polynomial on $\mathbb{R}^d$.
\end{prop}

\begin{proof}
Obviously, we can take $\widetilde{V}$ a vector subspace of $\mathbb{R}^d$ of dimension $d-1$ that
contains $V$, such that   $\widetilde{V}+\Lambda$ is not dense in $\mathbb{R}^d$.  
Observe that for every $h$ in $\mathbb{R}^d$, $\widetilde{V}+h$ is an affine subspace of $\mathbb{R}^d$
which is parallel to $\widetilde{V}$. Hence, if $w_0$ is a unitary normal vector to $\widetilde{V}$,
then for every $h$ in $\mathbb{R}^d$, $\widetilde{V}+h= \widetilde{V}+ s(h)w_0$ for a certain real number $s(h)$ which
depends on $h$. Indeed, if $P_{\widetilde{V}}:\mathbb{R}^d\to\mathbb{R}^d$ denotes the orthogonal projection of $\mathbb{R}^d$ onto $\widetilde{V}$, then every vector $h\in\mathbb{R}^d$ admits a unique decomposition $h=P_{\widetilde{V}}(h)+s(h)w_0$ and, obviously,
\[
\widetilde{V}+h=\widetilde{V}+(P_{\widetilde{V}}(h)+s(h)w_0) = \widetilde{V}+s(h)w_0
\]
Furthermore, as a direct consequence of the linearity of $P_{\widetilde{V}}$ and the uniqueness of the decompositions  $h=P_{\widetilde{V}}(h)+s(h)w_0$, we have that, for all $h_1,h_2\in\mathbb{R}^d$, $s(h_1+h_2)=s(h_1)+s(h_2)$. Hence  $\widetilde{V}+\Lambda$ can be decomposed as a union of sets $\bigcup_{s\in\Gamma} \widetilde{V}+ sw_0$ with $\Gamma= \{s(h):h\in\Lambda\}$  a  subgroup of $\mathbb{R}$ which is finitely generated since $\Lambda$ is finitely generated. Moreover, $\Gamma$ is not dense in $\mathbb{R}$ because  $\widetilde{V}+\Lambda$ is not dense in $\mathbb{R}^d$. It follows that $\Gamma= r\mathbb{Z}$ for certain positive real number $r$. Hence
\[
\widetilde{V}+\Lambda = \widetilde{V}+ rw_0\mathbb{Z}
\]
Take $h=rw_0$, $m=\min\{n_k\}_{k=1}^t$ and consider the function $f_m$ constructed in Example \ref{ejemp} with  $h_1=1$ and $h_2=2$. If we take $e(x)$ any exponential polynomial on $\mathbb{R}^d$, $\widetilde{H}=\tau(\mathbf{span}\{e\})$ the smallest translation invariant subspace of  $C(\mathbb{R}^d)$  which contains $\{e\}$,  and we define $\varphi(x+s h)=e(x)+f_m(s) $ for all $x\in \widetilde{V}$ and all $s\in\mathbb{R}$.
Let us consider the linear space
\[
H=\{f(P_{\widetilde{V}}(z)): f\in\widetilde{H}\}.
\]
Take $F(z)=f(P_{\widetilde{V}}(z))$ with $f\in\widetilde{H}$ and let $v\in \widetilde{V}$ be fixed. Then $$F(z+v)= f(P_{\widetilde{V}}(z+v)) =f(P_{\widetilde{V}}(z)+v)=(\tau_{v}f)(P_{\widetilde{V}}(z)) \in H$$
since $\tau_{v}f\in\widetilde{H}$, and
$$F(z+h)= f(P_{\widetilde{V}}(z+h)) =f(P_{\widetilde{V}}(z)) \in H$$
It follows that $\Delta_{h}(H) \subseteq H$ and  $\Delta_{v}(H) \subseteq H$ for all $v\in \widetilde{V}$. Consequently, if we define $p_k\in\mathbb{Z}$ by $p_kr=s(h_k)$, for $k=1,\cdots,t$, then
$$\bigcup_{k=1}^t \Delta_{h_k}(H) \subseteq H,$$
since
\begin{eqnarray*}
\Delta_{h_k}(H) &=&  \Delta_{P_{\widetilde{V}}(h_k)+s(h_k)w_0}(H) \\
&=& \Delta_{P_{\widetilde{V}}(h_k)}(\Delta_{p_kh}(H)) \\
&=&  \Delta_{P_{\widetilde{V}}(h_k)}(\Delta_{h}^{p_k}(H)) \subseteq H, k=1,2,\cdots,t.
\end{eqnarray*}
Furthermore, $\varphi$ satisfies that, if $z= P_{\widetilde{V}}(z)+sh$, then
\begin{eqnarray*}
\Delta_{h_k}^{m}\varphi(z) &=& \sum_{i=0}^{m}\binom{m}{i}(-1)^{m-i}\varphi(z+ih_k) \\
&=&  \sum_{i=0}^{m}\binom{m}{i}(-1)^{m-i}\varphi(P_{\widetilde{V}}(z)+sh+i(P_{\widetilde{V}}(h_k)+p_kh)) \\
&=&  \sum_{i=0}^{m}\binom{m}{i}(-1)^{m-i}e(P_{\widetilde{V}}(z)+iP_{\widetilde{V}}(h_k))+ \sum_{i=0}^{m}\binom{m}{i}(-1)^{m-i}f_m(s+ip_k)\\
&=&    \Delta_{P_{\widetilde{V}}(h_k)}^{m}e(P_{\widetilde{V}}(z)) + \Delta_{p_k}^{m}f_m(s)  \\
&=&    \Delta_{P_{\widetilde{V}}(h_k)}^{m}e(P_{\widetilde{V}}(z)) + (\Delta_{1}^{m})^{p_k}f_m(s)  \\ \\
&=&    \Delta_{P_{\widetilde{V}}(h_k)}^{m}e(P_{\widetilde{V}}(z)) \in H \text{ for } k=1,2,\cdots,t,
\end{eqnarray*}
since $P_{\widetilde{V}}(h_k)\in \widetilde{V}$, $e(P_{\widetilde{V}}(z)) \in H$ and $\Delta_{v}(H) \subseteq H$ for all $v\in \widetilde{V}$.
Thus, for every $k$ we have that $n_k\geq m$, $\Delta_{h_k}(H)\subseteq H$ and $\Delta_{h_k}^{m}\varphi\in H$, which obviously implies that
\[
\Delta_{h_k}^{n_k}\varphi= \Delta_{h_k}^{n_k-m}(\Delta_{h_k}^{m}\varphi)\in H.
\]
On the other hand, $\varphi$ is not an exponential polynomial on $\mathbb{R}^d$, since it is not an analytic function because $f_m$ is non-differentiable.
\end{proof}

We now give a description, for the case when $V$ is a proper subspace of $\mathbb{R}^d$ (arbitrary $d$), of the  sets of  continuous functions $f$  satisfying that $\Delta_{h_k}^{n_k}f\in H$  for $k=1,\cdots,t$, for a certain finite dimensional subspace $H$ of $C(\mathbb{R}^d)$ which is 
$\Delta_{h_k}^{m_k}$-invariant for some $m_k$, $k=1,\cdots, t$.



\begin{theorem}\label{cor_montel}
Let $t$ be a positive integer, let $n_1,n_2,\dots,n_t, m_1,m_2,\cdots,m_t$ be natural numbers, let $H\subseteq C(\mathbb{R}^d)$ be a linear finite dimensional subspace satisfying
$$\bigcup_{k=1}^t \Delta_{h_k}^{m_k}(H) \subseteq H.$$
Further let $f:\mathbb{R}^d\to\mathbb{R}$ be a continuous function satisfying
\[
\Delta_{h_k}^{n_k}f\in H
\]
for $k=1,\cdots,t$. If the subgroup $G$ in $\R^d$ generated by $\{h_1,h_2,\dots,h_t\}$ satisfies $\overline{G}=V\oplus \Lambda$, where $V$ is a vector subspace of $\mathbb{R}^d$, and $\Lambda$ is a discrete additive subgroup of $\mathbb{R}^d$, then for each $\lambda$ in $\Lambda$ there exist  a continuous exponential polynomial $e_{\lambda}:\R^d\to\R$ such that
\[
f(x+\lambda)=e_{\lambda}(x) \text{ for all } x\in V.
\]
\end{theorem}

For the proof of this result, we need firstly to introduce the following technical result:

\begin{lemma}\label{nuevo-ak}
Let $(G,+)$ be a commutative topological group, $h_1,\cdots,h_t\in G$ and  $G'=h_1\mathbb{Z}+\cdots+h_t\mathbb{Z}$. Let $H$ be a vector subspace of $\mathbb{C}^G$ such that $\tau_h(H)\subseteq H$ for all $h\in G'$, and assume that $f:G\to \mathbb{C}$ satisfies  $\Delta_{h_i}^{n_i}f\in H$ for certain natural numbers $n_i$ and $i=1,\cdots, t$. Take $N=n_1+\cdots+n_t$. Then  $\Delta_h^Nf\in H$ for all $h\in G'$. Moreover, if $H$ is a closed subspace of $C(G)$, then $\Delta_h^Nf\in H$ for all $h\in \overline{G'}$. 
\end{lemma}



\begin{proof} The proof follows similar arguments to those used in \cite[Theorem 2]{AS_AM} and, in fact, this lemma is a generalization of that result, which follows as a Corollary just imposing $H=\{0\}$. 

Take $N=n_1+\cdots +n_t$ and let $h\in G'$. Then there exist $m_1,\cdots,m_t\in\mathbb{Z}$ such that $h=m_1h_1+\cdots+m_th_t$ and
\begin{eqnarray*}
\Delta_h^Nf &=& \Delta_{m_1h_1+\cdots+m_th_t}^Nf=(\tau_{m_1h_1+\cdots+m_th_t}-1_d)^N(f)\\
&=& (\tau_{h_1}^{m_1}\tau_{h_2}^{m_2}\dots \tau_{h_t}^{m_t}-1_d)^{N}(f)\\
&=& \bigl[(\tau_{h_1}^{m_1}\tau_{h_2}^{m_2}\dots \tau_{h_t}^{m_t}-\tau_{h_2}^{m_2}\dots \tau_{h_t}^{m_t}) 
+ (\tau_{h_2}^{m_2}\dots \tau_{h_t}^{m_t}-\tau_{h_3}^{m_3}\dots \tau_{h_t}^{m_t}) \\
&\ & \ + (\tau_{h_3}^{m_3}\tau_{h_4}^{m_4}\dots \tau_{h_t}^{m_t}-\tau_{h_4}^{m_4}\dots \tau_{h_t}^{m_t})+
(\tau_{h_4}^{m_4}\dots \tau_{h_t}^{m_t}-\tau_{h_5}^{m_5}\dots \tau_{h_t}^{m_t}) \\
&\ & \  \cdots \\ 
&\ & \ + (\tau_{h_{t-2}}^{m_{t-2}}\tau_{h_{t-1}}^{m_{t-1}} \tau_{h_t}^{m_t}-\tau_{h_{t-1}}^{m_{t-1}}\tau_{h_t}^{m_t})+
(\tau_{h_{t-1}}^{m_{t-1}}\tau_{h_t}^{m_t}- \tau_{h_t}^{m_t})+(\tau_{h_t}^{m_t}-1_d)\bigr]^{N} (f) \\
&=& \bigl[(\tau_{h_1}^{m_1}-1_d)\tau_{h_2}^{m_2}\dots \tau_{h_t}^{m_t}+(\tau_{h_2}^{m_2}-1_d)\tau_{h_3}^{m_3}\dots \tau_{h_t}^{m_t}\\
&\ & \ \ +\dots+(\tau_{h_{t-1}}^{m_{t-1}}-1_d)\tau_{h_t}^{m_t}+ (\tau_{h_t}^{m_t}-1_d)\bigr]^{N}(f)\, .
\end{eqnarray*}
Expanding the $N$-th power, by the Multinomial Theorem, we obtain a sum of the form
\begin{equation} \label{multinomial}
\sum_{0\leq \alpha_1,\dots,\alpha_t\leq N} \frac{N!}{\alpha_1!\dots \alpha_t!} \prod_{i=1}^t  
(\tau_{h_{i+1}}^{m_{i+1}}\dots \tau_{h_t}^{m_t})^{\alpha_i}(\tau_{h_i}^{m_i}-1_d)^{\alpha_i}(f)\,,
\end{equation}
where the sum is also restricted by  $\alpha_1+\alpha_2+\dots+\alpha_t=N$, which implies that,  for some $k\in\{1,\cdots, t\}$,  $\alpha_k\geq n_k$.  In particular, for this concrete $k$, if $m_k>0$, we have that 
$$(\tau_{h_k}^{m_k}-1_d)^{\alpha_k}(f) = (\tau_{h_k}^{m_k}-1_d)^{\alpha_k-n_k}(\tau_{h_k}^{m_k}-1_d)^{n_k}(f) \in (\tau_{h_k}^{m_k}-1_d)^{\alpha_k-n_k}(H)\subseteq H ,$$ 
which follows from the assumption that $H$ is $G'$-invariant and the equation 
\begin{equation*}
(\tau_{h_k}^{m_k}-1_d)^{n_k}=(\tau_{h_k}^{m_k-1}+\tau_{h_k}^{m_k-2}+\dots+\tau_{h_k}+1_d)^{n_k} (\tau_{h_k}-1_d)^{n_k}\, .
\end{equation*}
If $m_k<0$, we can apply a similar argument, since we have that  
$$\tau_{h}^{-1}-1_d= \tau_{-h}-1_d=-\tau_{-h}(\tau_{h}-1_d)\, .$$
Consequently, all summands in \eqref{multinomial} belong to the vector space $H$, which proves the first claim of the lemma. If $H$ is a closed  subspace of $C(G)$ and $h\in \overline{G'}$ then there exist a sequence $\{g_n\}\subset G'$ converging to $h$ and the sequence given by  $f_n=\Delta_{g_n}^Nf$ converges to $\Delta_{h}^Nf$, which belongs to $H$ since $H$ is closed and $g_n\in H$ for all $n$. 
\end{proof}

\begin{proof}[Proof of Theorem \ref{cor_montel}]
Lemma \ref{dos} allow us to substitute $H$ satisfying $$\bigcup_{k=1}^t \Delta_{h_k}^{m_k}(H) \subseteq H.$$ by a finite dimensional subspace $\widetilde{H}$ of $C(\mathbb{R}^d)$ which contains $H$ and satisfies $$\bigcup_{k=1}^t \Delta_{h_k}(\widetilde{H}) \subseteq \widetilde{H}.$$
If we take $W=\mathbf{span}\{f\}$, we  can apply again Lemma \ref{dos} with $E=C(\mathbb{R}^d)$, $L_i=\Delta_{h_i}$, $i=1,\cdots,t$, to the vector space
$M=W+\widetilde{H}$, since
\begin{eqnarray*}
\Delta_{h_i}^{n_i}(W+\widetilde{H}) &=& \Delta_{h_i}^{n_i}(W)+ \Delta_{h_i}^{n_i}(\widetilde{H}) \\
&\subseteq&  H+\widetilde{H} \\
&\subseteq& \widetilde{H}\subseteq W+\widetilde{H},
\end{eqnarray*}
Hence
$M\subseteq Z=\diamond_{\Delta_{h_1},\Delta_{h_2},\cdots,\Delta_{h_t}}(M)$ and $Z$ is a finite dimensional subspace of $C(\mathbb{R}^d)$ satisfying $\Delta_{h_i}(Z)\subseteq Z$ , $i=1,2,\cdots,t$.  Let $V^{\perp}$ denote the orthogonal complement of $V$ in
$\mathbb{R}^d$ with respect to the standard scalar product and let $P_{V}:\mathbb{R}^d\to \mathbb{R}^d$ denote the orthogonal projection onto $V$ with respect to the standard scalar product of $\mathbb{R}^d$. We define the function $F:\R^d\to \R$ by $F(x)=f(P_V(x))$.  Obviously, $F$ is a continuous extension of $f_{|V}$.  Furthermore, if $\widetilde{H}$ admits a basis $\{g_k\}_{k=1}^m$, we introduce the new vector space
$\widetilde{\widetilde{H}}=\mathbf{span}\{G_k\}_{k=1}^m$, where $G_k(x)=g_k(P_V(x))$ for all  $x\in \mathbb{R}^d$. Then, if $x=v+w\in\mathbb{R}^d$ with $v\in V$, $w\in V^{\perp}$, $k\in \{1,\cdots,t\}$, and $j\in\{1,\cdots,m\}$, we have that
\begin{eqnarray*}
\Delta_{h_k}G_j(x) &=& G_j(v+w+h_k)-G_j(v+w)\\
&=& g_j(v+h_k)-g_j(v)\\
&=& \sum_{i=1}^m\alpha_{i,j}g_i(v) \text{, since } \Delta_{h_k}(\widetilde{H})\subseteq \widetilde{H}\\
&=& \sum_{i=1}^m\alpha_{i,j}G_i(v+w)\\
&=& \sum_{i=1}^m\alpha_{i,j}G_i(x).
\end{eqnarray*}
Thus, $\Delta_{h_k}(\widetilde{\widetilde{H}})\subseteq \widetilde{\widetilde{H}}$ for $k=1,\cdots,t$.

Take $\{h_1^*,\cdots,h_s^*\}\subseteq V^{\perp}$ such that
$\{h_1,\cdots,h_t,h_1^*,\cdots,h_s^*\}$ spans a dense subgroup of $\mathbb{R}^d$. Then if $x=v+w\in\mathbb{R}^d$ with $v\in V$, $w\in V^{\perp}$, $k\in \{1,\cdots,s\}$, and $j\in\{1,\cdots,m\}$, we have that
\begin{eqnarray*}
\Delta_{h_k^*}G_j(x) &=& G_j(v+w+h_k^*)-G_j(v+w)\\
&=& G_j(v)-G_j(v)=0.
\end{eqnarray*}
Hence we also have that $\Delta_{h_k^*}(\widetilde{\widetilde{H}})\subseteq \widetilde{\widetilde{H}}$ for $k=1,\cdots,s$. Anselone-Korevaar's theorem implies that all elements of $\widetilde{\widetilde{H}}$ are continuous exponential polynomials on $\mathbb{R}^d$.

Let us now do the computations for $F$. Take $N=n_1+\cdots+n_t$.  Then
\begin{eqnarray*}
&\ & \Delta_{h_i}^{N}F(x) = \sum_{k=0}^{N}\binom{N}{k}(-1)^{N-k}F(x+kh_i) \\
&=&  \sum_{k=0}^{N}\binom{N}{k}(-1)^{N-k}F(P_V(x)+kP_V(h_i)+[(x-P_V(x))+k(h_i-P_V(h_i))]) \\
& = & \sum_{k=0}^{N}\binom{N}{k}(-1)^{N-k}f(P_V(x)+kP_V(h_i)) \\
 &=& \Delta_{P_V(h_i)}^{N}f(P_V(x)) \\
 &=& \sum_{j=1}^m a_{i,j}g_j(P_V(x))  \text{ since }   \Delta_{P_V(h_i)}^{N}f\in \widetilde{H} \\
  &=& \sum_{j=1}^m a_{i,j}G_j(x) \in\widetilde{\widetilde{H}}.
\end{eqnarray*}
Here we have used that $\widetilde{H}$ is $G$-invariant and Lemma \ref{nuevo-ak} to conclude that $\Delta_{P_V(h_i)}^{N}f\in \widetilde{H}$ since $P_V(h_i)\in V\subseteq \overline{G}$ and $\widetilde{H}$ is a closed subspace of $C(\mathbb{R}^d)$, since it is finite dimensional.  On the other hand,  $\Delta_{h_j^*}F=0 \in\widetilde{\widetilde{H}}$ for $j=1,\cdots,s$. Hence we can apply Theorem \ref{EP} to $F$ to prove that $F$ is an exponential polynomial whose restriction to $V$ is $f_{|V}$. Thus, if we set $e_0=F$, we have that $e_0$ is a continuous exponential polynomial on $\mathbb{R}^d$ and $f(x)=e_0(x)$ for all $x$ in $V$.

Now let $\lambda$ be arbitrary in $\Lambda$ and we consider the function $g_{\lambda}:V\to\R$ defined by $g_{\lambda}(x)=f(x+\lambda)$ for $x$ in $V$ and $\lambda$ in $\Lambda$. Then $g_\lambda=\tau_{\lambda}(f)$, so that
$$\Delta_{h_i}^{n_i}g_{\lambda}= \Delta_{h_i}^{n_i}\tau_{\lambda}(f)=\tau_{\lambda}( \Delta_{h_i}^{n_i}(f))\in\tau_{\lambda}(H)\subseteq  \tau_{\lambda}(\widetilde{H}) \subseteq \widetilde{H}$$
Thus we can repeat all arguments above with $g_{\lambda}$ instead of $f$ to get that, for some continuous exponential polynomial $e_{\lambda}$ defined on $\mathbb{R}^d$ we have that
 $g_{\lambda}(x)=e_{\lambda}(x)$ for all $x\in V$. Hence, if $x\in V$ and $\lambda\in \Lambda$,
 \[
 f(x+\lambda)=g_{\lambda}(x)=e_{\lambda}(x).
 \]
This ends the proof.

\end{proof}

\end{document}